\documentclass[12pt]{article}\textwidth 160mm\textheight 235mm
\usepackage{amsfonts} \usepackage{amssymb} \usepackage{graphicx}
\usepackage{amsmath} \usepackage{amsthm} \usepackage{mathrsfs}
\usepackage{tikz} \usepackage{cancel} \usepackage{todonotes}
\usetikzlibrary{matrix} \usetikzlibrary{arrows,shapes}
\usepackage{cite} \usepackage{hyperref}
\usepackage{amsmath}
\usepackage{amsthm}
\usepackage{amsfonts}
\usepackage{multirow}
\usepackage{amssymb}
\usepackage{graphicx}
\usepackage{tcolorbox}
\usepackage{booktabs}
\begin{document}
\newtheorem{Assumption}{Assumption}
\newtheorem{Theorem}{Theorem}[section]
\newtheorem{Proposition}[Theorem]{Proposition}
\newtheorem{Remark}[Theorem]{Remark}
\newtheorem{Lemma}[Theorem]{Lemma}
\newtheorem{Corollary}[Theorem]{Corollary}
\newtheorem{Definition}[Theorem]{Definition}
\newtheorem{Example}[Theorem]{Example}
\newtheorem{Algorithm}[Theorem]{Algorithm}
\hypersetup{colorlinks=true,urlcolor=blue}
\expandafter\let\expandafter\oldproof\csname\string\proof\endcsname
\let\oldendproof\endproof
\renewenvironment{proof}[1][\proofname]{
\oldproof[\ttfamily\scshape \bf #1.]
}{\oldendproof}
\def\tilde{\widetilde}
\def\emp{\emptyset} \def\conv{{\rm conv}\,} \def\dist{{\rm dist}}
\def\dom{{\rm dom}\,} \def\span{{\rm span}\,} \def\epi{{\rm epi\,}}
\def\rge{{\rm rge\,}} \def\om{\ominus} 
\def\lra{\longrightarrow} \def\Lra{\Longrightarrow}
\def\Lm{{\Lambda}} \def\d{{\rm d}} \def\sub{\partial} \def\B{\mathbb
B} \def\L{{\mathscr{L}}} \def\o{\overline} \def\oRR{\overline{\RR}}
\def\ox{\overline{x}} \def\ov{\overline{y}} \def\oz{\overline{z}}
\def\cl{{\rm cl}\,} \def\Ra{\Rightarrow} \def\disp{\displaystyle}
\def\tto{\rightrightarrows} \def\Hat{\widehat}
\def\Tilde{\widetilde} \def\Bar{\overline}
\def\pa{\partial}
\def\ra{\rangle}
\def\la{\langle} \def\ve{\varepsilon}
\def\e{\varepsilon}
\def\ox{\bar{x}} 
\def\oy{\bar{y}} \def\ov{\bar{v}} \def\oz{\ox}
\def\col{\colon\!\!\!\!} \def\ov{\bar{v}} \def\ow{\bar{w}}
\def\ou{\bar{u}} \def\ot{\bar{t}} \def\op{\bar{p}} \def\oq{\bar{q}}
\def\co{\mbox{\rm co}\,} \def\cone{\mbox{\rm cone}\,}
\def\ri{\mbox{\rm ri}\,} \def\inte{\mbox{\rm int}\,}
\def\gph{\mbox{\rm gph}\,} \def\epi{\mbox{\rm epi}\,}
\def\dim{\mbox{\rm dim}\,} \def\dom{\mbox{\rm dom}\,}
\def\ker{\mbox{\rm ker}\,} \def\proj{\mbox{\rm proj}\,}
\def\lip{\mbox{\rm lip}\,} \def\reg{\mbox{\rm reg}\,}
\def\aff{\mbox{\rm aff}\,} \def\clco{\mbox{\rm clco}\,}
\def\cl{\mbox{\rm cl}\,} \def\rank{\mbox{\rm rank}\,}
\def\h{\hfill\triangle} \def\dn{\downarrow} \def\O{\Omega}
\def\ph{\varphi} \def\emp{\emptyset} \def\st{\stackrel}
\def\oR{\Bar{\R}} \def\lm{\lambda} \def\gg{\gamma} \def\dd{\delta}
\def\al{\alpha} \def\Th{\Theta} \def\N{{\rm I\!N}}
\def\R{{\rm I\!R}}
\def\th{\theta} \def\vt{\vartheta} \def\e{\varepsilon} \def \b{{\}_{k\in\N}}}
\def\Limsup{\mathop{{\rm Lim}\,{\rm sup}}}
\def\Liminf{\mathop{{\rm Lim}\,{\rm inf}}}
\def\Lim{\mathop{{\rm Lim}\,{\rm }}}
\def\sm{\hbox{${1\over
2}$}} \def\varepsilonrl{\;\rule[-0.4mm]{0.2mm}{0.27cm}\;}
\def\varepsilonrll{\;\rule[-0.7mm]{0.2mm}{0.37cm}\;}
\def\varepsilonrlm{\;\rule[-0.5mm]{0.2mm}{0.33cm}\;}
\def\Limsup{\mathop{{\rm Lim}\,{\rm sup}}}
\newcommand{\hoac}[1]{
\left[\begin{aligned}#1\end{aligned}\right.}
\newcommand{\heva}[1]{
\left\{\begin{aligned}#1\end{aligned}\right.}
\numberwithin{equation}{section}
\newenvironment{sol}{\paragraph{Solution.}}{\hfill$\square$}
\newenvironment{cm}{\paragraph{Proof.}}{\hfill$\square$}
\title{\bf
Relationships between Global and Local Monotonicity of Operators}

\author{P.D. Khanh\footnote{Department of Mathematics, Ho Chi Minh City University of Education, Ho Chi Minh City, Vietnam. E-mail: pdkhanh182@gmail.com, khanhpd@hcmue.edu.vn.}\;\;V.V.H. Khoa\footnote{Department of Mathematics, Wayne State University, Detroit, Michigan, USA. E-mail: khoavu@wayne.edu. Research of this author was partly supported by the  US National Science Foundation under grant DMS-2204519.}\;\;J.E. Mart\'inez-Legaz \footnote{Departament d’Economia i d’Hist\`oria Econ\`omica, Universitat Aut\`onoma de Barcelona, and BGSMath, Barcelona, Spain. Email: JuanEnrique.Martinez.Legaz@uab.cat. Research of this author was partly supported by grant PID2022-136399NB-C22 from MICINN, Spain, and by ERDF, ``A way to make Europe”, European Union.}\;\;B.S. Mordukhovich\footnote{Department of Mathematics, Wayne State University, Detroit, Michigan, USA. E-mail: aa1086@wayne.edu. Research of this author was partly supported by the  US National Science Foundation under grant DMS-2204519, by the Australian Research Council under Discovery Project DP190100555, and by the Project 111 of China under grant D21024.} \\[3ex]
{\bf Dedicated to Terry Rockafellar, with admiration}}
\maketitle

\noindent
{\small{\bf Abstract}. The paper is devoted to establishing relationships between global and local monotonicity, as well as their maximality versions, for single-valued and set-valued mappings between finite-dimensional and infinite-dimensional spaces. We first show that for single-valued operators with convex domains in locally convex topological spaces, their continuity ensures that their global monotonicity agrees with the local one around any point of the graph. This also holds for set-valued mappings defined on the real line under a certain connectedness condition. The situation is different for set-valued operators in multidimensional spaces as demonstrated by an example of locally monotone operator on the plane that is not globally monotone. Finally, we invoke coderivative criteria from variational analysis to characterize both global and local maximal monotonicity of set-valued operators in Hilbert spaces to verify the equivalence between these monotonicity properties under the closed-graph and global hypomonotonicity assumptions.
\\[1ex]

\section{Introduction}\label{sec:intro}
\vspace*{-0.05in}

It has been well recognized in mathematics, especially in the area of nonlinear, convex, and variational analysis, that the notion of {\em globally monotone} for single-valued and set-valued operators, and particularly of its {\em maximality} version, plays a crucial role in the study and applications of a variety of theoretical and numerical issues. The contributions of Terry Rockafellar to both of these issues are monumental. We also refer the reader to, e.g., the books \cite{BC,Mordukhovich18,phelps,Simons08,Rockafellar98,thibault} and the extended bibliographies therein for various aspects of global monotonicity of operators and their numerous applications in finite-dimensional and infinite-dimensional spaces.

In the area of variational analysis, which mainly deals with {\em local} behaviors of functions and mappings, the global versions of monotonicity  and its maximality are  usually superfluous. In many contexts, one would require monotonicity behavior only within a neighborhood of some graphical point, and that is when {\em local monotonicity} arises.

Similarly to global monotonicity, the {\em local maximal monotonicity} of operators plays a central role in variational theory and applications. But in contrast to global maximal monotonicity, the two versions of local maximal monotonicity of set-valued mappings/multifunctions have been introduced and exploited in the variational analysis and optimization literature. The first one goes back to the paper by Poliquin and Rockafellar \cite{poli}, where it appeared in the study of {\em tilt stability} of local minimizers. The second notion of local maximal monotonicity should be credited to Pennanen \cite{Pen02} who employed it in developing powerful convergence results for the proximal point and related algorithms of finite-dimensional optimization. It has been recently proved in 
\cite{KKMP23} that both versions of local maximal monotonicity of multifunctions are equivalent in reflexive Banach spaces admitting an equivalent norm with some geometric properties. Various applications of local maximal monotonicity to both theoretical and numerical aspects of variational analysis and optimization can be found in, e.g., \cite{KKMP23,KKMP24,kmp23,mor24,Mordukhovich18,MordukhovichNghia1,rtr251,rtr257,roc23} and other recent publications. Among the most important applications of local monotonicity and its maximality version for subgradient operators, we mention {\em variational convexity} of extended-real-valued functions, the notion introduced by Rockafellar in \cite{rtr257} and then largely investigated and employed in variational analysis and optimization-related areas in \cite{KKMP23,KKMP24,mor24,rtr257,roc23,ding}.\vspace*{0.03in}

Very natural questions arise on establishing relationships between {\em local and global monotonicity} of single-valued and set-valued operators and between their {\em maximality} versions. It is not for the first time when relationships between local and global properties are under investigation in mathematics; in particular, for problems of convex analysis. For instance, the famous theorem obtained independently by Tietze \cite{Tietze} and Nakajima \cite{Nakajima} in 1928 states that any closed, connected, and locally convex subset of an Euclidean space is necessarily convex. 

The goal of this paper is to establish both positive and negative results concerning (maximal) local and global monotonicity of single-valued and set-valued operators in finite and infinite dimensions. By {\em closing the gap} between the global and local monotonicity properties of operators, as well as between the corresponding maximality versions, we understand clarifying whether the global (maximal) monotonicity of the operator in question is equivalent to its local (maximal) monotonicity property around any point of the operator graph.

First we close the gap between the local and global monotonicity, even without maximality, for {\em continuous single-valued} operators with {\em convex domains} in locally convex topological spaces. Then we show that, while the above equivalence holds for {\em univariate set-valued mappings} $T\colon\R\tto\R$ with {\em path-connected graphs}, it {\em fails} for higher-dimensional operators $T\colon\R^2\tto\R^2$. 

A rather surprising result is obtained for the {\em maximal monotonicity} of set-valued operators in Hilbert spaces. Involving powerful tools of {\em generalized differentiation} in variational analysis, we prove that the local maximal monotonicity of operators with {\em closed graphs} agrees with the global one provided the fulfillment of the {\em global hypomonotonicity} property.\vspace*{0.03in} 

The rest of our paper is organized as follows. Section~\ref{sec:monotone} recalls the notions of global and local monotonicity of set-valued mappings, together with their maximality versions, and discusses in more details the local maximal monotonicity of operators in reflexive Banach spaces. 

In Section~\ref{sec:single}, we establish the equivalence between local and global monotonicity of continuous single-valued mappings defined on locally convex topological vector spaces and present two independent proofs of the obtained equivalence.

Sections~\ref{sec:set} and \ref{sec:not global} are devoted to the relationships between local and global monotonicity of set-valued operators defined on finite-dimensional spaces. While the main result of Section~\ref{sec:set} solves the equivalence between the local and global monotonicity in the affirmative for univariate operators with path-connected graphs, Section~\ref{sec:not global} presents a counterexample for such an equivalence in the case of a multifunction with the path-connected graph in $\R^4$.  

Section~\ref{sec:max} addresses maximal monotonicity of set-valued operators defined in Hilbert spaces. Based on the powerful coderivative criteria for the global and local maximal monotonicity of closed-graph mappings, we reveal the equivalence between these notions under global hypomonotonicity, which is shown to be essential for the obtained equivalence. The final Section~\ref{sec:conc} presents concluding remarks and discusses some unsolved problems for the future research.\vspace*{0.03in}

Throughout the paper, we often consider the $n$-dimensional space $\R^n$ with the usual inner product $\la x,y\ra :=\sum_{i=1}^n x_i y_i$ for $x=(x_1,\ldots,x_n)$ and $y=(y_1,\ldots,y_n)$, and the induced norm $\|x\|:= \sqrt{\la x,x\ra}$. Regarding a topological vector space $X$ and its dual $X^*$, the value of a functional $x^*\in X^*$ at $x\in X$ is denoted by $\la x^*,x \ra$. Given a set-valued mapping $T: X\rightrightarrows X^*$, denote by
\begin{eqnarray*}
\dom T:=\big\{x\in X\;\big|\;T(x)\ne\emp\big\}\;\mbox{ and }\;\gph T:=\big\{(x,x^*)\in X\times X^*\;\big|\;x^*\in T(x)\big\}
\end{eqnarray*}
its {\em domain} and {\em graph}, respectively. The mapping $T$ is said to be \textit{proper} if $\dom T \neq \emp$, which is always assumed in what follows.

\section{Global and Local Monotonicity of Multifunctions}\label{sec:monotone}

In this section, we recall and discuss the definitions of global and local monotonicity of set-valued mappings/operators/multifunctions together with their maximality version. For a proper operator $T:X \rightrightarrows X^*$ defined on a topological vector space, it is said that:

{\bf(i)} $T$ is {\em monotone} on $X$ if 
\begin{eqnarray}\label{eq:monotone}
\la x^*-y^*, x- y \ra \ge 0 \;\text{ whenever }\; (x,x^*),(y,y^*)\in \gph T.
\end{eqnarray}

{\bf(ii)}  $T$ is {\em maximal monotone} on $X$ if \eqref{eq:monotone} is fulfilled and moreover, the graph of $T$ cannot be properly contained in the graph of any other monotone operator, i.e., $\gph T = \gph S$ for any monotone mapping $S:X\rightrightarrows X^*$ satisfying $\gph T \subset \gph S$.\vspace*{0.05in}

Upon localizing the  global monotonicity notion, we restrict the fulfillment of \eqref{eq:monotone} within a neighborhood of a graph point as follows.

{\bf(iii)} Given $(u,u^*)\in \gph T$, $T$ is  \textit{locally monotone around} $(u,u^*)$ if there exists a neighborhood $W\subset X\times X^*$ of $(u,u^*)$ such that 
\begin{eqnarray}\label{eq:local-monotone}
\la x^*-y^*, x- y \ra \ge 0 \;\text{ for all }\; (x,x^*),(y,y^*)\in \gph T \cap W.
\end{eqnarray}
Observe that the local monotonicity of set-valued mappings enjoys the following {\em robustness property}: if $T$ is locally monotone around a graph point $(u,u^*)$ corresponding to a neighborhood $W$, then the local monotonicity of $T$ around any point $(x,x^*)$ in $W$ is fulfilled. It is said in this case that $T$ is {\em locally monotone relative to} the neighborhood $W$. 

Regarding local {\em maximal} monotonicity, there exist the following two versions, which are used in the literature; see \cite{KKMP23} for more discussions.

{\bf(iv)} $T$ is {\em locally maximal monotone of type $(A)$} around $(u,u^*)\in \gph T$ if there exists a neighborhood $W$ of $(u,u^*)$ such that $T$ is monotone with respect to $W$ and that $\gph T\cap W=\gph S\cap W$ for any mapping $S:X\rightrightarrows X^*$, which is monotone with respect to $W$ and satisfies the inclusion $\gph T\cap W\subset \gph S \cap W$.

{\bf(v)} $T$ {\em locally maximal monotone of type $(B)$} around $(u,u^*)\in \gph T$ if there exist a maximal monotone mapping $\overline{T}:X\rightrightarrows X^*$ and a neighborhood $W$ of $(u,u^*)$ for which 
\begin{equation*}
\gph \overline{T}\cap W = \gph T \cap W.
\end{equation*}

It is not difficult to observe that if $T:X\rightrightarrows X^*$ is locally maximal monotone of type $(A)$ around $(u,u^*)$, then $T$ is locally maximal monotone of type $(B)$ around this point.
As follows from \cite[Theorem~3.3]{KKMP23}, both local maximal monotonicity types (A) and (B) agree if $X$ is a {\em reflexive Banach space}. Indeed, each reflexive Banach space admits an equivalent norm satisfying the requirements of \cite[Theorem~3.3]{KKMP23}, while the formulations of local maximal monotonicity properties in (iv) and (v) do not depend on renorming.

In the rest of the paper, we use---for the general setting of locally convex topological spaces---only the local maximal monotonicity notion of type (B), without indicating the type of local maximal monotonicity of multifunctions. 

\section{Relationships between Global and Local Monotonicity of Single-Valued Operators}\label{sec:single}

This section deals with the monotonicity of operators $T\colon X\supset\dom T\to X^*$, which are defined on locally convex topological vector spaces while being {\em single-valued on their domain}. The main result of this section is as follows.

\begin{Theorem}\label{theo:single}
Let $X$ be a locally convex topological vector space. Assume that  $T\colon\dom T\rightarrow X^*$ has a convex domain and that $T$ is continuous relative to any segment in $\dom T$. Then the local monotonicity of $T$ around any graph point is equivalent to the global monotonicity of $T$ on $X$.      
\end{Theorem}

We present two completely independent proofs of Theorem~\ref{theo:single} that are of their own interest. \\\\
{\bf Proof.} Fix $p,q\in \dom T$ with $p\neq q$ and get $[p,q]\subset \dom T$, where $[p,q]:=\{\lambda p + (1-\lambda)q\mid 0\le \lambda \le 1\}$ is the segment connecting the points $p$ and $q$. On the segment $[p,q]$, we consider the binary relation ``$\sim$" defined by $x\sim y$ if and only if
\begin{equation*}
 \la T(x_1)-T(x_2), x_1-x_2\ra \ge 0\;\mbox{ whenever }\;x_1,x_2\in[x,y].
\end{equation*}
This relation is clearly reflexive and symmetric. To check its transitivity, take distinct vectors $x,y,z$ such that $x\sim y$ and $y\sim z$. We may suppose without loss of generality that 
$$
y\in ]x,z[:=\big\{\lambda x + (1-\lambda)z\;\big|\;\lambda \in(0,1)\big\}.
$$
Indeed, taking $x\in ]y,z[$ gives us $[x,z]\subset [y,z]$, and hence $x\sim z$. On the other hand, it follows from $z\in ]x,y[$ that $[x,z]\subset[x,y]$, and so $x\sim z$. Picking now $x_1,x_2\in [x,z]$, it suffices to proceed with the case where $x_1\in [x,y[$ and $x_2\in ]y,z]$. Then 
\begin{equation*}
\la T(x_1)-T(x_2),x_1-x_2\ra =\dfrac{\|x_1-x_2\|}{\|x_1-y\|} \la T(x_1)-T(y),x_1-y\ra +\dfrac{\|x_1-x_2\|}{\|y-x_2\|} \la T(y)-T(x_2),y-x_2\ra \ge 0, 
\end{equation*}
which means that $x\sim z$, and ``$\sim$" is an equivalence relation on $[p,q]$. Furthermore, we claim that the relation ``$\sim$" splits $[p,q]$ into open equivalence classes. To verify this, fix $x\in [p,q]$ and let $y\in[x]$, where the latter notation stands for the equivalence class of $x$. It follows from the assumed local monotonicity of the operator $T$ and the local convexity of the space $X$ that there exists a neighborhood $U\times W \subset X\times X^*$ of $(y,T(y))$ for which $U$ is convex and
\begin{equation}\label{theo2-monotone}
T \text{ is monotone relative to }U\times W.
\end{equation}
The continuity of $T$ relative to $[p,q]$ allows us to shrink the convex set $U$, if needed, so that 
\begin{equation}\label{theo2-useinner}
T(U\cap [p,q])\subset W.
\end{equation}
We claim that $U\cap [p,q]\subset [x]$, which would justify the openness of the class $[x]$ in $[p,q]$. To furnish this, we take any $z\in U\cap [p,q]$ and show that $z\sim y$, ensuring therefore that $z\in [x]$ by $y\sim x$. To this end, pick arbitrary vectors $x_1,x_2\in [z,y]$ and observe that having $z,y\in U\cap [p,q]$ yields $x_1,x_2\in U\cap [p,q]$ by the convexity of the neighborhood $U$ of $y$. Employing now \eqref{theo2-monotone} and \eqref{theo2-useinner} tells us that $\la T(x_1)-T(x_2),x_1-x_2\ra \ge 0$, i.e., $z\sim y$. Taking finally into account that $y\in [x]$, we get $z\in [x]$, which  verifies that $[x]$ is open in $[p,q]$. 

If there are more than one equivalence class, then $[p,q]$ is the union of two open disjoint subsets $[p]$ and $\bigcup_{x\nsim p}[x]$. This contradicts the connectedness of $[p,q]$ and shows that there exists only one equivalence class $[p]$. Therefore, $p\sim q$, which completes the first proof of the theorem.\\
\\
{\bf Alternate proof.} Fix $x, y \in \dom T$, and for any $\lambda \in[0,1]$ denote $x_\lambda:=(1-\lambda) x+\lambda y$. Since $\operatorname{dom} T$ is convex, we have $x_\lambda \in \operatorname{dom} T$. Therefore, the set
$$
\Lambda:=\big\{\lambda \in[0,1]\;\big|\;\langle T\left(x_\lambda)-T(y),x-y\right\rangle \ge 0\big\}
$$
is well-defined. Moreover, it has the following properties.
\begin{itemize}

\item Since $\la T(y)-T(y),x-y\ra \ge 0$ and $y=x_{1}$, we have that $1\in \Lambda$, and thus $\Lambda$ is nonempty.

\item The set $\Lambda$ is closed. Indeed, take a sequence $\{\lambda_k\}\subset \Lambda$ such that $\lambda_k \rightarrow \lambda$ as $k\to\infty$. It is obvious that $\lambda \in [0,1]$. Having further
\begin{equation*}
\big\la T\big((1-\lambda_k)x+\lambda_k y\big) -T(y),x-y\big\ra \ge 0\;\mbox{ for all }\;k\in \N
\end{equation*}
and $(1-\lambda_k)x+\lambda_k y \xrightarrow{[x,y]} (1-\lambda)x+\lambda y$ as $k\rightarrow \infty$, we deduce from the continuity of $T$ restricted to the segment $[x,y]$ that
\begin{equation*}
\big\la T\big( (1-\lambda)x +\lambda y\big)-T(y),x-y\big\ra\ge 0,
\end{equation*}
and therefore $\lambda\in \Lambda$, which shows that $\Lambda$ is closed.
\end{itemize}

Define now $\bar{\lambda}:=\min \Lambda$ and get $\bar{\lambda} \in \Lambda$ due to the closedness of $\Lambda$. Suppose that $\bar{\lambda}>0$, i.e., $x_{\bar{\lambda}}\neq x$. Since $\bar{\lambda}\in \Lambda$, we obtain that
\begin{equation*}
\big\langle T(x_{\bar{\lambda}})-T(y),x-y\big\rangle \ge 0.
\end{equation*}
It follows from the local monotonicity of $T$ around $\big(x_{\bar{\lambda}},T(x_{\bar{\lambda}})\big)$ that there exists a neighborhood $U\times W \subset X\times X^*$ of $\big(x_{\bar{\lambda}},T(x_{\bar{\lambda}})\big)$ for which 
\begin{equation}\label{theo2-monotone2}
T \text{ is monotone relative to }U\times W.
\end{equation}
By the continuity of $T$ relative to $[x,y]$, we can always select $U$ so that 
\begin{equation}\label{theo2-useinner2}
T(U\cap [x,y])\subset W.
\end{equation}
Representing $x_{\bar{\lambda}-\varepsilon}-x_{\bar{\lambda}}=\varepsilon(x-y)$ gives us the net convergence $x_{\bar{\lambda}-\varepsilon}\rightarrow x_{\bar{\lambda}}$ in $X$ as $\varepsilon \rightarrow 0$. Since $U$ is a neighborhood of $x_{\bar{\lambda}}$, for small $\varepsilon \in ] 0, \bar{\lambda}[$ we have that $x_{\bar{\lambda}-\varepsilon} \in U\cap [x,y]$ and $T\big(x_{\bar{\lambda}-\varepsilon}\big)\in W$ due to \eqref{theo2-useinner2}. Therefore, it follows from \eqref{theo2-monotone2} that
$$
\left\langle T\left(x_{\bar{\lambda}-\varepsilon}\right)-T\left(x_{\bar{\lambda}}\right), x_{\bar{\lambda}-\varepsilon}-x_{\bar{\lambda}}\right\rangle \geq 0,
$$
which can be equivalently rewritten as
$$
\left\langle T\left(x_{\bar{\lambda}-\varepsilon}\right)-T\left(x_{\bar{\lambda}}\right), x-y\right\rangle \geq 0
$$
while implying in turn that $\left\langle T\left(x_{\bar{\lambda}-\varepsilon}\right)-T(y), x-y\right\rangle \geq 0$. The latter means that $\bar{\lambda}-\varepsilon \in \Lambda$, which contradicts the definition of $\bar{\lambda}$. This allows us to conclude that $\bar{\lambda}=0$, i.e., $0 \in \Lambda$. Remembering that $x_0=x$ tells us that $\langle T(x)-T(y), x-y\rangle \geq 0$, which verifies the global monotonicity of $T$ on $X$ as claimed in the theorem. $\h$

\begin{Remark}\label{rem:single}\rm 
When $\dom T\subset X$ is convex (or more generally, path-connected as defined in 
Section~\ref{sec:set}), the continuity of $T$ relative to its domain immediately yields the path-connectedness (arc-connectedness if $X$ is Hausdorff) of $\gph T$. The reverse implication, although being true when $X=\R$ (see \cite[Example~4, Section~22]{ross}), is false in multidimensional spaces by the counterexample of $T:\R^2 \rightarrow \R^2$ given by
\begin{equation*}
T(x,y):= \begin{cases}
(0,0) &\text{if }\;x<0, \\
(x,0) &\text{if }\;x\ge 0,\;y\ge 0,\\
(-x,0) &\text{if }\;x\ge 0,y<0.
\end{cases}
\end{equation*}
\end{Remark}\vspace*{0.1in}

The following example of two parts shows that both assumptions in Theorem~\ref{theo:single} are {\em essential} for the fulfillment of the equivalence result.

\begin{Example}\rm 
\quad 
\begin{enumerate}

\item The convexity assumption on the domain is not superfluous in Theorem~\ref{theo:single}. Indeed, define the mapping $T:\R \setminus \{0\}\rightrightarrows \R$ by $T(x):=\{0\}$ when $x<0$, and $T(x):=\{-1\}$ when $x>0$. It is clearly continuous relative to its domain and is locally monotone around any graph point. Nevertheless, the global monotonicity of $T$ obviously fails.
 
\item The imposed continuity of the operator in Theorem~\ref{theo:single} is essential. For instance, the function $T\colon\R\to\R$ defined by $T(x):=0$ when $x\le 0$ and $T(x):=-1$ when $x>0$ clearly serves as a counterexample.
\end{enumerate}
\end{Example}

\section{Monotonicity of Univariate Set-Valued Operators}\label{sec:set}

The main goal of this section is to investigate relationships between global and local monotonicity of univariate set-valued operators $T\colon\R\tto\R$. Recall first that a subset $S$ of a topological space $X$ is {\em path-connected} (resp.\ {\em arc-connected}) if any two points $x\neq y$ of $S$ can be joined by a path (resp.\ an arc) within $S$, i.e., there exists $\varphi:[0,1]\rightarrow S$ such that $\varphi (0)=x$, $\varphi (1)=y$, and $\varphi:[0,1]\rightarrow \varphi([0,1])$ is continuous (resp.\ homeomorphic). Note that when $X$ is Hausdorff (in particular, $X=\R^n$), the notions of path-connectedness and arc-connectedness coincide.\vspace*{0.05in} 

In the next three propositions, we discuss some monotonicity and extremality properties  of real-valued components of {\em paths} in $\R^2$, i.e., continuous mappings from $[0,1]$ to $\R^2$.

\begin{Proposition}\label{prop}
Let $\varphi=(\varphi_1,\varphi_2):[0,1]\rightarrow \R^2$ be an injective continuous mapping. Assume that for all $\overline{t}\in [0,1]$, we have the inclusion
\begin{equation}\label{y-mono}
\varphi (s) \in \big(\varphi(t)-\R^2_+\big) \cup \big( \varphi(t)+\R^2_+\big)\;\text{ whenever }\; s,t \in [0,1]\;\mbox{ are near }\;\overline{t}.
\end{equation}
Then the following assertions are satisfied:

{\bf(i)} If $\varphi_1 (0)<\varphi_1 (1)$, then $\varphi_1(0)=\min\varphi_1 ([0,1])$, $\varphi_1(1)=\max\varphi_1 ([0,1])$, and  
\begin{equation}\label{contra-min}
\varphi (t) \in \varphi(0)+\R^2_+ \text{ for all } t\in ]0,1] \text{ near } 0, \quad \varphi (t) \in \varphi(1)-\R^2_+ \text{ for all } t\in [0,1[ \text{ near } 1.
\end{equation}

{\bf(ii)} If $\varphi_1 (0)=\varphi_1 (1)$, then $\varphi_1$ is constant on $[0,1]$.
\end{Proposition}
\begin{proof}
To verify (i), define the numbers 
\begin{equation}\label{t1}
z_1:=\min\varphi_1 ([0,1]),\; z_1^*:= \min\{z^*\in \R\mid (z_1,z^*)\in \varphi([0,1])\},\;\text{ and }\; t_1:=\varphi^{-1}(z_1,z_1^*).
\end{equation}
If $t_1=1$, then $\varphi (t_1)=\varphi (1)$, i.e., $(z_1,z_1^*)=\big(\varphi_1 (1),\varphi_2 (1)\big)$. On the other hand, it follows from $\varphi_1 (0)<\varphi_1 (1)$ that $z_1=\min\varphi_1 ([0,1])<\varphi_1 (1)$, which contradicts the previous assertion. Thus $t_1\in [0,1[$. Now we are going to show that $t_1=0$. Suppose on the contrary that $t_1>0$, and select by \eqref{y-mono} such $0<\e<\min\{t_1,1-t_1\}$ that $[t_1-\e,t_1+\e]\subset ]0,1[$ and that 
\begin{equation}\label{y-mono*}
\varphi (s)\in \big( \varphi(t)-\R^2_+\big)\cup \big(\varphi (t)+\R^2_+\big)\;\text{ whenever }\;s,t\in [t_1-\e,t_1+\e].
\end{equation}
In particular, \eqref{y-mono*} implies that $\varphi (t_1-\e)\in \big( \varphi(t_1+\e)-\R^2_+\big)\cup \big( \varphi (t_1+\e)+\R^2_+\big)$. In what follows, we consider the case where
\begin{equation}\label{-and+}
\varphi (t_1-\e)\in \varphi(t_1+\e)-\R^2_+,\;\text{ i.e., }\;\varphi (t_1+\e)\in \varphi (t_1-\e)+\R^2_+,
\end{equation}
while observing that the case $\varphi (t_1-\e)\in \varphi (t_1+\e)+\R^2_+$ can be treated similarly.

Note that the inclusion in \eqref{y-mono*} ensures that 
\begin{equation}\label{use-w}
\varphi (t_1)\in \big( \varphi(t_1-\e)-\R^2_+\big)\cup \big( \varphi (t_1-\e)+\R^2_+\big),
\end{equation}
and recall that $\varphi_1 (t_1)=\min\varphi_1 ([0,1]) \le \varphi_1 (t_1-\e)$. If $\varphi_1 (t_1)<\varphi_1 (t_1-\e)$, then it follows from \eqref{use-w} that $\varphi_2 (t_1)\le \varphi_2 (t_1-\e)$, and thus $ \varphi (t_1)\in \varphi (t_1-\e)-\R^2_+$. If $z_1=\varphi_1 (t_1)=\varphi_1 (t_1-\e)$, then $\varphi_2 (t_1)\le  \varphi_2 (t_1-\e)$ by the definition of $z_1^*$ in \eqref{t1}. This yields 
\begin{equation}\label{t+}
\varphi (t_1)\in \varphi (t_1-\e)-\R^2_+.
\end{equation}
Observe that the set $\big( \varphi ( t_1-\varepsilon ) -
\R_{+}^{2}\big)\cup{\big( \varphi ( t_1-\varepsilon ) +
\R_{+}^{2}\big)} $ is connected while the set $\big( \varphi ( t_1-\varepsilon ) -\R
_{+}^{2}\big) \cup \big( \varphi ( t_1-\varepsilon) +\R_{+}^{2}\big)\setminus \{ \varphi ( t_1-\varepsilon )\} $ is not. Combining \eqref{y-mono*}, \eqref{-and+}, \eqref{t+} with the continuity of $\varphi$, we find $\widetilde{t}\in \left[ t_1 ,t_1+\e\right] $ such that $\varphi( \widetilde{t}) =\varphi ( t_1-\varepsilon)$. This
contradicts the injectivity of $\varphi $ and show that $t_1=0$, and therefore $\varphi_1 (0)=\varphi_1 (t_1)=\min \varphi_1 ([0,1])$. Then the first assertion in \eqref{contra-min} is a direct consequence of \eqref{y-mono} and the definition of $\varphi(0)=(z_1,z_1^*)$. The other assertions in (i) involving $\varphi (1)$ are deduced directly from those for $\varphi(0)$ by replacing $(\varphi_1,\varphi_2)$ by $(-\varphi_1,-\varphi_2)$.\vspace*{0.05in}

Now we verify (ii). Arguing by contradiction, suppose that there exists 
$t_{0}\in ]0,1[$ such that $\varphi _{1}( t_{0})
\neq \varphi _{1}( 0)$. For definiteness, we assume
that $\varphi _{1}( t_{0}) <\varphi _{1}( 0)$ while observing that the other case where
$\varphi _{1}( t_{0}) >\varphi _{1}( 0)$ can be treated similarly. Using again definitions \eqref{t1} and taking into account that $\varphi _{1}( 
t_1) =z_1\leq \varphi _{1}( t_{0}) <\varphi
_{1}( 0) =\varphi _{1}( 1)$ give us $t_1\in ]0,1[$, and so we proceed as in the proof of (i) arriving in this way to a contradiction with the injectivity of $\varphi$. Therefore, $\varphi _{1}$ is constant on $[0,1]$, which concludes the proof of the proposition.
\end{proof}\vspace*{0.05in}

It turns out that the technical assumption in \eqref{y-mono} is equivalent to the {\em local monotonicity} of the set $\varphi([0,1])$ as in the following proposition.

\begin{Proposition}\label{equiv}
In the setting of Proposition~{\rm\ref{prop}}, condition \eqref{y-mono} holds for a fixed number $\overline{t}\in [0,1]$ if and only if the set $\varphi([0,1])$ is locally monotone around $\varphi(\overline{t})$ in the sense that there exists a neighborhood $W$ of $\varphi (\overline{t})$ such that
\begin{equation*}
(x_1-x_2)(y_1-y_2)\ge 0 \quad \text{for all}\quad (x_1,y_1),(x_2,y_2)\in W \cap \varphi ([0,1]).
\end{equation*}
\end{Proposition}
\begin{proof}
The ``if" part is a direct consequence of the continuity of the mapping $\varphi$. To verify the converse implication, suppose 
on the contrary that $\varphi ([0,1])$ is not locally
monotone around $\varphi (\overline{t})$. Then for every $k\in\N$,
the set $\varphi ([0,1])\cap B\left( \varphi (\overline{t}),\frac{1}{k}\right)$, where $B\left( \varphi (\overline{t}),\frac{1}{k}\right)$ stands for the open ball centered at $\varphi (\overline{t})$ with radius $\frac{1}{k}$, is
not monotone. The latter means that there exist numbers $s_{k},t_{k}\in \lbrack 0,1]$ such that for all $k\in\N$ we have
\begin{equation}
\varphi (s_{k}),\varphi \left( t_{k}\right) \in B\left( \varphi (\overline{t}
),\frac{1}{k}\right),\label{incl}
\end{equation}
\begin{equation}
\varphi (s_{k})\notin \left( \varphi (t_{k})-\R_{+}^{2}\right) \cup
\left( \varphi (t_{k})+\R_{+}^{2}\right).  \label{nonmon}
\end{equation}%
It follows from \eqref{y-mono} that there exists $\varepsilon >0$ such that $\varphi (s)\in
\left( \varphi (t)-\R_{+}^{2}\right) \cup \left( \varphi (t)+\R_{+}^{2}\right) $ for all $s,t\in \lbrack 0,1]\cap \big]\overline{t}-\varepsilon ,\overline{t}+\varepsilon \big[$. Then relation (\ref{nonmon}) implies that
either $\left\vert s_{k}-\overline{t}\right\vert \geq \varepsilon $, or $
\left\vert t_{k}-\overline{t}\right\vert \geq \varepsilon $, or both conditions hold. Without loss of generality, assume that $\left\vert s_{k}-\overline{t}\right\vert
\geq\varepsilon $ and that the sequence $\left\{ s_{k}\right\} $ is
convergent. Denoting $s:=\lim s_{k}$, we clearly get that $\left\vert s-\overline{t}
\right\vert \geq \varepsilon$, and so $s\neq $ $\overline{t}$. On the other hand, it follows  from (\ref{incl}) and the continuity of $\varphi$ that $\varphi \left(
s\right) =\varphi \left( \overline{t}\right)$. This contradicts the
injectivity of $\varphi$ and thus completes the proof of the proposition.
\end{proof}

The final proposition shows that in the setting of Proposition~\ref{prop}, the functions $\ph_1$ and $\ph_2$ are {\em globally monotone} on $[0,1]$. 

\begin{Proposition}\label{mon comp} Let $\varphi_1,\ph_2$ be taken from Proposition~{\rm\ref{prop}} under the assumptions imposed therein. Then both $\varphi_{1}$ and $\varphi _{2}$ are monotone on $[0,1]$.
\end{Proposition}
\begin{proof}
We only verify the monotonicity of $\varphi _{1}$ while observing that the proof for $
\varphi _{2}$ is similar by replacing $\varphi $ with 
$( \varphi _{2},\varphi _{1})$. In the case where $\varphi _{1}(0)
=\varphi _{1}( 1)$, it follows from Proposition~\ref{prop}(ii) that the function $\varphi
_{1}$ is constant on $[0,1]$ and hence it is monotone. Consider now the case where $
\varphi _{1}( 0)\ne\varphi _{1}( 1)$ and assume for definiteness that
$\varphi _{1}( 0) <\varphi _{1}( 1)$ and show that $\varphi _{1}$ is nondecreasing. Suppose the contrary and then find $t_{1},t_{2}\in \left[ 0,1\right] $ such that $0\le t_{1}<t_{2}$ and  that $\varphi
_{1}( t_{1}) >\varphi _{1}( t_{2})$. Since $\varphi_1(0)<\varphi_1 (1)$, Proposition~\ref{prop} implies that $\varphi_1(0)=\min\varphi_1 ([0,1])$, and thus we get $\varphi_1 (0)\le \varphi_1 (t_2)<\varphi_1 (t_1)$. Employing again
Proposition~\ref{prop} with $\varphi$ replaced by $\eta: =( \eta _{1},\eta _{2})
:\left[ 0,1\right] \rightarrow \R^2$ with $\eta ( t)
:=\varphi ( t_{2}t)$, we deduce that
\begin{equation*}
\max \eta _{1}( \left[ 0,1\right] )  =\max\{\eta_1 (0),\eta_1(1)\}=\max\{\varphi_1(0),\varphi_1 (t_2)\} <\varphi
_{1}( t_{1})  =\eta _{1}\left( \frac{t_{1}}{t_{2}}\right) ,
\end{equation*}%
which is nonsense. Thus $\varphi _{1}$ is nondecreasing as claimed. If $\varphi _{1}( 0) >\varphi_{1}( 1)$, then replacing $\varphi $ with the mapping $\psi :\left[ 0,1\right] \rightarrow
\R^2$ defined by $\psi ( t) :=\varphi ( 1-t)$ tells us that
$\varphi _{1}$ is nonincreasing on $[0,1]$, which completes the proof of the proposition.
\end{proof}

Using the above propositions, we are now ready to obtain the main result of this section.

\begin{Theorem}\label{main}
Consider a set-valued mapping $T:\R \rightrightarrows \R$ whose graph is
path-connected. If $T$ is locally monotone around every point of
its graph, then $T$ is globally monotone on $\R$.
\end{Theorem}
\begin{proof}
Given pairs $( x,x^{\ast }),( y,y^{\ast }) \in \gph T$ with $x<y$, consider a continuous function
$\varphi =( \varphi _{1},\varphi _{2})\colon[0,1]\to \gph T$
with $\varphi ( 0) =( x,x^{\ast})$ and $\varphi ( 1) =( y,y^{\ast })$. Such a function $\varphi$ can be selected injective, since pathwise connectedness and arcwise
connectedness are equivalent for subsets of Euclidean spaces; see, e.g., \cite[Corollary~31.6]{eng}. We intend to show that $x^{\ast }\leq y^{\ast }$ via arguing by contradiction, i.e., assuming that $x^{\ast }>y^{\ast }$.

Recall from Proposition~\ref{equiv} that the local monotonicity of $T$ around every point of its graph yields the fulfillment of \eqref{y-mono} for any $\overline{t}\in [0,1]$. Applying Proposition~\ref{prop}(i) to $\varphi=(\varphi_1,\varphi_2):[0,1]\rightarrow \gph T$ with $\varphi_1(0)=x<y=\varphi_1 (1)$ gives us 
\begin{equation}\label{contra1}
\varphi(t) \notin \varphi (1) +\R^2_+\;\text{ for all }\;t\in [0,1[\;\text{ near }\;\bar t=1.
\end{equation}
On the other hand, we have for the function $\psi=(\psi_1,\psi_2):= (-\varphi_2,-\varphi_1)$ that $\psi_1 (0)=-x^*<-y^* = \psi_1 (1)$. Thus it follows from Proposition~\ref{prop}(i) that $\psi (t) \notin \psi (1) +\R^2_+\;\text{ for all }\;t\in [0,1[\;\text{ near }\;\bar t=1,$ i.e., we get
\begin{equation}\label{contra2}
\varphi (t)\notin \varphi(1) -\R^2_+\;\text{ whenever }\;t\in [0,1[\;\text{ near }\;\bar t=1.
\end{equation}
Since $\varphi:[0,1]\rightarrow \gph T$ is continuous, the conditions in \eqref{contra1} and \eqref{contra2} contradict the local monotonicity of $T$ around $\varphi (1)\in \gph T$, which verifies the claim of the theorem.
\end{proof}

\begin{Remark}\rm 
Recall from classical real analysis (see, e.g.,  \cite[Section~3.6]{Basic-RA}) that for a real-valued function $f:\R \rightarrow \R$, $f$ is nondecreasing on $\R$ if and only if it is nondecreasing around every point $x_0\in \R$, i.e., for every $x_0\in \R$ there exists $\delta>0$ such that $f(x)\le f(x_0)\le f(y)$ whenever $x\in ]x_0-\delta,x_0[$ and $y\in ]x_0,x_0+\delta[$. 

In the case of univariate single-valued mappings $T:\R\rightarrow \R$, the imposed path-connectedness of $\gph T$ in Theorem~\ref{main} is equivalent to the continuity of $T$ itself; see, e.g., \cite[Example~4, Section~22]{ross}. Under continuity, the local monotonicity of $T$ around $(x_0,T(x_0))$ automatically implies its nondecreasing property around $x_0$, and thus our results in Theorem~\ref{theo:single} and Theorem~\ref{main} are weaker than the classical one when restricting ourselves to single-valued function from $\R$ to $\R$. On the other hand, Theorem~\ref{theo:single} and Theorem~\ref{main} hold in the framework of topological vector spaces and set-valued mappings, respectively.
\end{Remark}

\begin{Remark} \rm Observe that
Theorem~\ref{main} {\em fails} if the path-connectedness of $\gph T$ is replaced by the convexity of $\dom T$, which is strictly weaker in the univariate setting under consideration. Indeed, consider the operator $T:\R \rightrightarrows \R$ defined by 
\begin{equation}\label{1}
T(x):= \begin{cases}
\{0\} &\text{if }\;x\le 0, \\
\{-1\} &\text{if }\;x > 0
\end{cases}
\end{equation}
whose graph is depicted as follows:
\begin{center}
\includegraphics[scale=0.6]{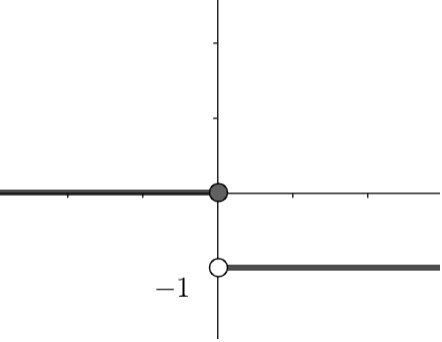}
\end{center}
In this case, $\dom T=]-\infty,\infty[$ is convex and $T$ is locally monotone around any point in its graph. Nevertheless, $T$ is not globally monotone.
\end{Remark}

\section{Local Monotonicity Is Not Global on the Plane}\label{sec:not global}

In this section, we show that the result of Theorem~\ref{main} cannot be extended to multidimensional settings by constructing a set-valued operator with the path-connected graph in $\R^4$, which is locally monotone around any point of its graph while not being globally monotone on $\R^2$. 

\begin{Example}\label{exam:Legaz}
There exists an operator $T:\R^2 \rightrightarrows \R^2$ such that $\gph T$ is path-connected, $T$ is locally monotone around any point on its graph, but $T$ is not globally monotone on $\R^2$.
\end{Example}
\begin{proof}
Consider the set-valued mapping $T:\R^2 \rightrightarrows \R^2 $ defined by 
\begin{equation}\label{T}
T(x,y):=
\begin{cases}
\Big\{\big(0,\max\{3x-1,0\}\big)\Big\} &\mbox{ if }\;x\ge 0,\;y=0, \\
\{(-y,0)\} &\mbox{ if }\;x=0,\;y\ge 0,\\
\emptyset &\mbox{ otherwise}.
\end{cases}
\end{equation}
The graph of $T$ is split into two parts of the space $\R^4$, which is the collection of quadruples $(x,y,z,t)$. The first part of $\gph T$ is represented by 
$$
T(x,0)=\Big\{\big(0,\max\{3x-1,0\}\big)\Big\},\quad x\ge 0,
$$
and lies within the $xt$-plane as shown below:
\begin{center}
\includegraphics[scale=0.6]{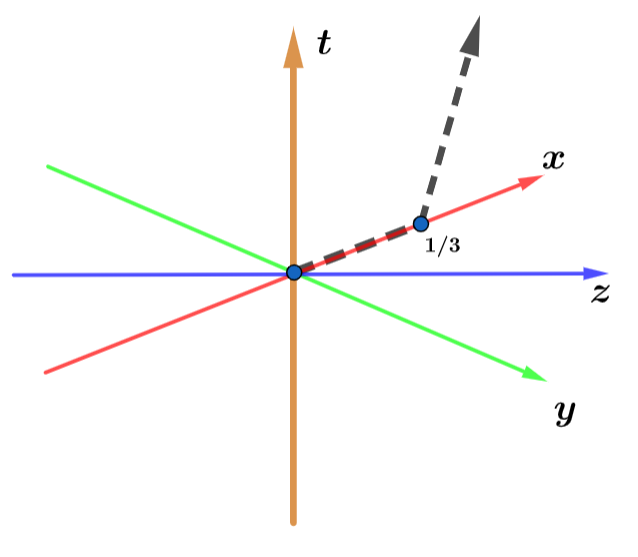}
\end{center}
The second part of the graph is given by
$$
T(0,y)=\{(-y,0)\},\quad y\ge 0
$$
staying inside the $yz$-plane as follows:
\begin{center}
\includegraphics[scale=0.6]{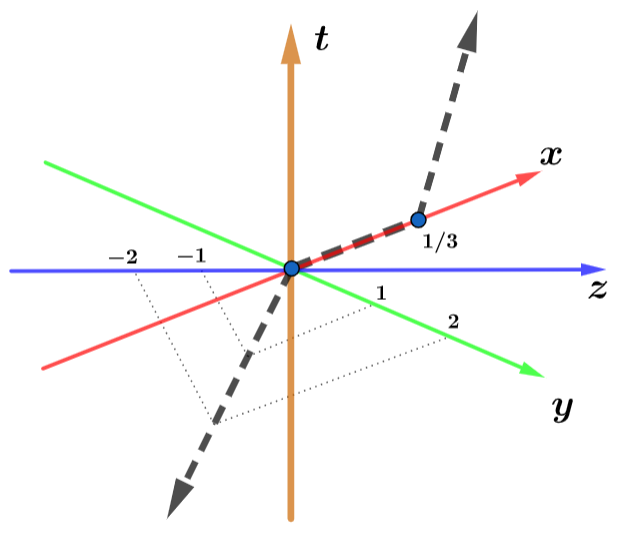}
\end{center}

It is clear to see that the set $\gph T$ is path-connected. Let us show that $T$ is locally monotone around any point of its graph. This is furnished in the following three steps:
\quad
\begin{enumerate}

\item[\bf Step~1.] First we check that $T$ is monotone relative to the open set $W_1:= ]0,\infty[ \times \R \times \R \times \R$. Indeed, consider $\big((x_i,y_i),(z_i,t_i)\big)\in \gph T \cap W_1$, $i=1,2$, and deduce from the definition of $T$ that
\begin{equation*}
(x_i,y_i,z_i,t_i) = (x_i,0,0,t_i), \quad i=1,2,
\end{equation*}
with $x_i>0$. This tells us therefore that
\begin{equation*}
\langle (x_1,y_1)-(x_2,y_2),(z_1,t_1)-(z_2,t_2)\rangle = \big\langle (x_1-x_2,0),(0,t_1-t_2) \big\rangle = 0.
\end{equation*}
The given calculations ensure the monotonicity of $T$ relative to $W_1$ and also the local monotonicity of $T$ around any point in $\gph T$ with the positive $x$-coordinate.

\item[\bf Step~2.] Here we show that $T$ is monotone relative to the open set $W_2:=\R\times ]0,+\infty[ \times \R \times \R$. To proceed, consider the pairs $\big((x_i,y_i),(z_i,t_i)\big)\in \gph T \cap W_2$, $i=1,2$, and deduce from the definition of $T$ that
\begin{equation*}
(x_i,y_i,z_i,t_i) = (0,y_i,z_i,0), \quad i=1,2,
\end{equation*}
with $y_i>0$. Thus we arrive at the equalities 
\begin{equation*}
\la (x_1,y_1)-(x_2,y_2),(z_1,t_1)-(z_2,t_2)\ra = \big\la (0,y_1-y_2),(z_1-z_2,0) \big\ra = 0.
\end{equation*}
The obtained calculations tell us that the operator $T$ is locally monotone around any point $(x,y,z,t)\in \gph T$ with $y>0$.

\item[\bf Step~3.] Let us check next that $T$ is monotone relative to the (open) neighborhood $W_3:=B_{1/3}(0)\times \R \times \R \times \R$ of $(0,0,0,0)$. Considering the pairs $\big((x_i,y_i),(z_i,t_i)\big)\in \gph T \cap W_3$, $i=1,2$, we clearly have that $x_1,x_2\ge 0$ and $t_1=t_2=0$. If $x_1=x_2$, then 
\begin{equation*}
\big\la (x_1,y_1)-(x_2,y_2),(z_1,t_1)-(z_2,t_2)\ra = \la (0,y_1-y_2),(z_1-z_2,0)\big\ra =0.
\end{equation*}
Otherwise, it remains to consider the case where $x_1>x_2\ge 0$. If $x_2>0$, then it reduces to Step~1. For $x_1>0$ and $x_2=0$, we get
\begin{equation*}
 (x_1,y_1,z_1,t_1)=(x_1,0,0,0),\ \text{ and } \ (x_2,y_2,z_2,t_2)= (0,y_2,-y_2,0),\;\;y_2\ge 0.
 \end{equation*}
Therefore, we arrive at the condition
\begin{equation*}
\la (x_1,y_1)-(x_2,y_2),(z_1,t_1)-(z_2,t_2)\ra =\la (x_1,-y_2),(y_2,0)\ra = x_1y_2 \ge 0.
\end{equation*}
This guarantees the monotonicity of $T$ relative to $W_3$, and thus the local monotonicity of $T$ around $(0,0,0,0)$.
\end{enumerate}

However, $T$ is not globally monotone. This is easily seen via the computation
\begin{equation*}
\la (1,0)-(0,1),T(1,0)-T(0,1)\ra = \la (1,0)-(0,1),(0,2)-(-1,0)\ra = \la (1,-1),(1,2)\ra = 1-2=-1<0.
\end{equation*}
\end{proof}

\section{Relationships between Local and Global Maximal Monotonicity in Hilbert Spaces}\label{sec:max}

This section addresses the {\em maximality} aspects of global and local monotonicity of set-valued operators in Hilbert spaces. The approach developed here is completely different from our derivations in the previous sections. We are now based on the machinery of {\em variational analysis} and {\em generalized differentiation} of set-valued mappings and employ the powerful {\em coderivative criteria} for global and local maximal monotonicity.

First we recall some tools of variational analysis used in what follows; see, e.g., the books \cite{Mordukhovich18,mor24,Rockafellar98} for more details and references in finite and infinite dimensions.

Given a nonempty set $\O$ of a Hilbert space $X$, the {\em regular normal cone} to $\O$ at $x\in\O$ is
\begin{equation*}
\Hat{N}(x;\Omega):=\Big\{x^*\in X\;\Big|\;\limsup_{u\overset{\Omega}{\rightarrow}x}\frac{\langle
x^*,u-x\rangle}{\|u-x\|}\le 0\Big\}. 
\end{equation*} 

The {\em regular coderivative} of a set-valued mapping $F\colon X\tto X$ at $(\ox,\oy)\in\gph F$ is defined as 
\begin{equation}\label{reg-cod} 
\Hat
D^*F(\ox,\oy)(y^*):=\big\{x^*\in X\;\big|\;(x^*,-y^*)\in\Hat N((\ox,\oy);\gph F)\big\},\quad y^*\in X.
\end{equation}

A multifunction $T:X\rightrightarrows X$ is called (globally) {\em hypomonotone} if there exists a constant $r>0$ such that the mapping $T+rI$ is monotone on $X$, i.e., 
\begin{equation}\label{hypo}
\la y_1-y_2,x_1-x_2\ra \ge -r\|x_1-x_2\|^2 \quad \text{for all}\quad (x_1,y_1),(x_2,y_2)\in \gph T.
\end{equation}
Being strictly weaker than hypomonotonicity, the {\em local hypomonotonicity} property of $T:X\rightrightarrows X$ around $(x,y)\in \gph T$ refers to the fulfillment of \eqref{hypo} within some neighborhood of $(x,y)$.\vspace*{0.05in}

Now we formulate the following two coderivative characterizations of global and local maximal monotonicity on which our equivalence result below is based. Since the framework of this section is Hilbert spaces, there is no difference between the local maximal monotonicity of types (A) and  (B), which are referred to as ``local maximal monotonicity."

\begin{Lemma}\label{lem:code-max1}{\rm {\bf (see \cite[Theorem~3.2]{CBN})}} Let $T:X\rightrightarrows X$ be a set-valued mapping of closed graph. Then we have the equivalent assertions:

{\bf(i)} $T$ is maximal monotone on $X$.

{\bf(ii)}  $T$ is globally hypomonotone on $X$ with some modulus $r>0$, and the regular coderivative $\widehat{D}^* T(u,v)$ is positive-semidefinite on $X$ in the sense that
\begin{eqnarray*}
\la z,w \ra \ge 0\;\text{ for any }\;z\in \widehat{D}^* T(u,v)(w),\;(u,v)\in\gph T,\;\mbox{ and }\;w\in X.
\end{eqnarray*}
\end{Lemma}

\begin{Lemma}\label{lem:code-max2}
{\rm {\bf (see \cite[Theorem~6.4]{KKMP23})}}
Let $T:X\rightrightarrows X$ be of locally closed graph around $(\ox,\oy)\in \gph T$. Then the following assertions are equivalent:

{\bf(i)} $T$ is locally maximal monotone around $(\ox,\oy)$.

{\bf(ii)} $T$ is locally hypomonotone around $(\ox,\oy)$ with some modulus $r>0$, and there exists a neighborhood $W$ of $(\ox,\oy)$ such that
\begin{eqnarray*}
\la z,w\ra \ge 0\ \text{ whenever }\ z\in \Hat{D}^* T(u,v)(w)\ \text{ and } \ (u,v)\in W\cap \gph T,\; w\in X.
\end{eqnarray*}
\end{Lemma}

Utilizing these lemmas, we arrive at the equivalence between global and local maximal monotonicity of closed-graph operators under the global hypomonotonicity assumption. 

\begin{Theorem}\label{theo:max}
Let $T:X\rightrightarrows X$ be a globally hypomonotone set-valued mapping defined on a Hilbert space $X$. Assume that the graph of $T$ is closed. Then the following assertions are equivalent:

{\bf(i)} $T$ is maximal monotone on $X$.

{\bf(ii)} $T$ is locally maximal monotone around any point in its graph.
\end{Theorem}
\begin{proof}
Implication (i)$\Longrightarrow$(ii) obviously follows from the definitions. To verify (ii)$\Longrightarrow$(i), we get from Lemma~\ref{lem:code-max2} due to the local monotonicity of $T$ around any $(x,y)\in\gph T$ that 
\begin{equation}\label{code-charac}
\la z,w\ra \ge 0 \ \text{ whenever }\ z\in \widehat{D}^* T(u,v)(w) \ \text{ and }\ (u,v)\in \gph T,\; w\in X.
\end{equation}
Then Lemma~\ref{lem:code-max1} implies by \eqref{code-charac} and the assumed graph-closedness and global hypomonotonicity of $T$ that the operator $T$ is maximal monotone on $X$ as claimed.
\end{proof}

Let us demonstrate that all the assumptions of Theorem~\ref{theo:max} are essential for the fulfillment of the equivalence conclusion even for simple operators from $\R$ to $\R$. 

\begin{Remark}\rm 
\quad 
\begin{enumerate}
\item The graph-closedness assumption in Theorem~\ref{theo:max} cannot be dropped. Indeed, the mapping $T:\R \rightrightarrows \R$ defined by $T(x):=\{0\}$ if $-1<x<1$ and $T(x):=\emptyset$ otherwise, is globally monotone and locally maximal monotone around any point of its graph. However, $T$ is not maximal monotone since $\gph T$ is properly contained in $\gph S=\gph T \cup \{(1,0)\}$, which is monotone. This happens because the set $\gph T=]-1,1[\times \{0\}$ is not closed in $\R^2$.

\item Implication (ii)$\Longrightarrow$(i) in Theorem~\ref{theo:max} fails when $T$ is not hypomonotone on $X$. For instance, consider the multifunction
\begin{equation*}
T(x):=\begin{cases}
\R &\text{if }x\in\{0,1\}, \\
\emptyset &\text{otherwise.}
\end{cases}
\end{equation*}
Although $\gph T=\{0,1\}\times \R$ is closed and the local maximal monotonicity of $T$ around any graph point holds, we have that $T+rI=T$ is not monotone for any $r>0$, and thus $T$ is not hypomonotone. Consequently, it cannot be maximal monotone. 
\end{enumerate}    
\end{Remark}

Finally, we present a consequence of Theorems~\ref{theo:max} and \ref{theo:single} for operators on Hilbert spaces that are single-valued on their domains.

\begin{Corollary}\label{coro:combine} Let $T\colon \dom T\to X$ be of convex domain and closed graph. Assume in addition 
$T$ is continuous relative to any segment in $\dom T$. Then we have the equivalent assertions:

{\bf(i)} $T$ is maximal monotone.

{\bf(ii)} $T$ is locally maximal monotone around any point in its graph.
\end{Corollary}
\begin{proof}
By the definitions given in Section~\ref{sec:monotone}, it suffices to verify implication (ii)$\Longrightarrow$(i). Since $\dom T$ is convex and $T$ is continuous relative to any segment in its domain, the assumed local maximal monotonicity of $T$ around any graph point implies that $T$ is globally monotone on $X$ via Theorem~\ref{theo:single}. In particular, $T$ is globally hypomonotone. Employing then Theorem~\ref{theo:max} given that $\gph T$ is closed, we immediately get the maximal monotonicity of the mapping $T$.
\end{proof}

\section{Conclusions and Future Research}\label{sec:conc}

This paper contributes to the study of local monotonicity of set-valued mappings and to establishing its relationships with global monotonicity including the maximality versions. In particular, we tackle the problem of closing the gap between local and global (maximal) monotonicity. To this end, some verifiable sufficient conditions are presented in rather nice frameworks. On the flip side, it is worth-mentioning that Example~\ref{exam:Legaz} clearly distinguishes between local monotonicity and global monotonicity in the framework of multidimensional Euclidean spaces. Thus the study of local monotonicity in such cases is indispensable.

In the future research, we aim to extend our results obtained for single-valued mappings to general settings of multifunctions, as well as to relax the imposed continuity assumptions on operators. In particular, it is still an open question whether one can replace the continuity of $T$ in Theorem~\ref{theo:single} by the path-connectedness (or stronger, arc-connectedness) of $\gph T$.\vspace*{0.1in}

{\bf Acknowledgements}: The authors thank Dat Ba Tran for his valuable discussions.

\end{document}